\documentclass[hjml]{amsart}
\usepackage{amssymb,latexsym}
\newtheorem{theorem}{Theorem}
\newtheorem{lemma}{Lemma}

\newtheorem*{theorema} {Theorem A (Wolff)}

\begin{document}

\title[Wolff's Problem of Ideals]{Wolff's Problem of Ideals in the Multiplier Algebra on Weighted Dirichlet
Space}

\author{Debendra P. Banjade and Tavan T. Trent}

\address{Department of Mathematics\\
 The University of Alabama\\
 Box 870350\\
 Tuscaloosa, AL 35487-0350\\
 (205)758-4275}

\email{dpbanjade@crimson.ua.edu, ttrent@as.ua.edu}


\begin{abstract}
We establish an analogue of Wolff's theorem on ideals in $H^{\infty}(\mathbb{D})$
for the multiplier algebra of weighted Dirichlet space.
\end{abstract}

\maketitle
\section{Introduction}

In this paper we wish to extend a theorem of Wolff, concerning ideals in $H^{\infty}(\mathbb{D})$,  to the setting
of multiplier algebras on weighted Dirichlet spaces. Our techniques will closely follow those used in Banjade-Trent [BT] for
the (unweighted) Dirichlet space.  The new material requires the boundedness of a certain singular integral operator (Lemma 3) and the boundedness of the Beurling transform (Lemma 4) on some $L^2$ spaces with weights.

In 1962 Carleson {[}C{]} proved his famous \lq\lq Corona theorem''
characterizing when a finitely generated ideal in $H^{\infty}(\mathbb{D})$
is actually all of $H^{\infty}(\mathbb{D})$. Independently, Rosenblum
{[}R{]}, Tolokonnikov {[}To{]}, and Uchiyama gave an infinite version
of Carleson's work on $H^{\infty}(\mathbb{D})$. In an effort to classify
ideal membership for finitely-generated ideals in $H^{\infty}(\mathbb{D})$,
Wolff {[}G{]} proved the following version:

\begin{theorema} If
\begin{align}
\{f_{j}\}_{j=1}^{n}\subset H^{\infty}(\mathbb{D}),H\in H^{\infty}(\mathbb{D})\quad\text{and}\notag\\
|H(z)|\le\left(\sum_{j=1}^{n}\,|f_{j}(z)|^{2}\right)^{\frac{1}{2}}\;\;\text{for all }\, z\in\mathbb{D},
\end{align}
 then
\[
H^{3}\in\mathcal{I}(\{f_{j}\}_{j=1}^{n}),
\]

\smallskip
\noindent
the ideal generated by $\{f_{j}\}_{j=1}^{n}$ in $H^{\infty}(\mathbb{D})$.
\end{theorema} 
It is known that (1) is not, in general, sufficient
for $H$ itself or even for $H^{2}$ to be in $\mathcal{I}(\{f_{j}\}_{j=1}^{n})$ , see Rao's example in Garnett {[}G{]} and Treil {[}T{]}.

\medskip{}
For the algebra of multipliers on Dirichlet space, the analogue of Wolff's ideal theorem was established by the authors in {[}BT{]}.
Since the analogue of the corona theorem for the algebra of multipliers
on weighted Dirichlet space was established in Kidane-Trent $\left[\mbox{KT}\right]$, it seems plausible that Wolff-type ideal results should be extended to
the algebra of multipliers on weighted Dirichlet space. This is what we intend to do in this paper.

\medskip
For $\alpha\in(0,1)$, we use $\mathcal{D}_{\alpha}$ to denote the weighted Dirichlet space on the
unit disk, $\mathbb{D}$. That is,
\begin{multline*}
\mathcal{D}_{\alpha}=\{\, f:\;\mathbb{D}\rightarrow\mathbb{C}\mid\; f\text{ is analytic on }\mathbb{D}\text{ and for }f(z)=\overset{\infty}{\underset{n=0}{\sum}}a_{n}\, z^{n},\\
\left\Vert f\right\Vert _{\mathcal{D}_{\alpha}}^{2}=\underset{n=0}{\overset{\infty}{\sum}}\,(n+1)^{\alpha}\left|a_{n}\right|^{2}<\infty\}.
\end{multline*}

\medskip
We will use other equivalent norms for smooth functions in $\mathcal{D}_{\alpha}$
as follows,
\[
\left\Vert f\right\Vert _{\mathcal{D}_{\alpha}}^{2}=\int_{-\pi}^{\pi}\vert f\vert^{2}d\sigma+\int_{D}\vert f\,'(z)\vert^{2}\left(1-\vert z\vert^{2}\right)^{1-\alpha}dA(z)\qquad\mbox{ and}
\]

\[
\left\Vert f\right\Vert _{\mathcal{D}_{\alpha}}^{2}=\int_{-\pi}^{\pi}\vert f\vert^{2}d\sigma+\int_{-\pi}^{\pi}\int_{-\pi}^{\pi}\frac{\vert f(e^{it})-f(e^{i\theta})\vert^{2}}{\vert e^{it}-e^{i\theta}\vert^{1+\alpha}}\, d\sigma d\sigma.
\]

For ease of notation, we will denote $\left(1-\vert z\vert^{2}\right)^{1-\alpha}dA(z)$
by $dA_{\alpha}(z).$
Also, we will consider $\overset{\infty}{\underset{1}{\oplus}}\,\mathcal{D}_{\alpha}$
as an $l^{2}$-valued weighted Dirichlet space. The norms in this
case are exactly as above but we will replace the absolute value by
$l^{2}$-norms. Moreover, we use $\mathcal{HD_{\alpha}}$ to denote
the harmonic weighted Dirichlet space (restricted to the boundary of $\mathbb{D}$).
The functions in $\mathcal{D}_{\alpha}$ have only vanishing negative
Fourier coefficients whereas the functions in $\mathcal{\mathcal{HD_{\alpha}}}$
may have negative Fourier coefficients which do not vanish. Again,
if $f$ is smooth on $\partial \mathbb{D}$, the boundary of the unit disk
$\mathbb{D}$, then
$$
\left\Vert f\right\Vert _{\mathcal{HD_{\alpha}}}^{2}=\int_{-\pi}^{\pi}\vert f\vert^{2}\, d\sigma+\int_{-\pi}^{\pi}\int_{-\pi}^{\pi}\frac{\vert f(e^{it})-f(e^{i\theta})\vert^{2}}{\vert e^{it}-e^{i\theta}\vert^{1+\alpha}}\, d\sigma d\sigma.$$

We use $\mathcal{M}(\mathcal{D}_{\alpha})$ to denote the multiplier
algebra of weighted Dirichlet space, defined as: $\mathcal{M}(\mathcal{D}_{\alpha})=\left\{ \phi\in\mathcal{D}_{\alpha}:\,\phi f\in\mathcal{D}_{\alpha}\;\text{ for all }f\in\mathcal{D}_{\alpha}\right\} ,$
and we will denote the multiplier algebra of harmonic weighted Dirichlet space
by $\mathcal{M}(\mathcal{HD_{\alpha}})$, defined similarly (but only on $\partial \mathbb{D}$). Also, we will use $\mathcal{M}_{l^{2}}(\mathcal{D}_{\alpha})$ to denote the  multiplier algebra of $l^2$ - valued weighted Dirichlet space.

Given $\left\{ f_{j}\right\} _{j=1}^{\infty}\subset\mathcal{M}(\mathcal{D}_{\alpha})$,
we consider $F(z)=\left(f_{1}(z),\, f_{2}(z),\dots\right)$ for $z\in \mathbb{D}.$
We define the row operator $M_{F}^{R}:\overset{\infty}{\underset{1}{\oplus}}\,\mathcal{D}_{\alpha}\rightarrow\mathcal{D}_{\alpha}$
\, by
$$
M_{F}^{R}\left(\left\{ h_{j}\right\} _{j=1}^{\infty}\right)=\overset{\infty}{\underset{j=1}{\sum}}\, f_{j}h_{j}\,\mbox{ for }\left\{ h_{j}\right\} _{j=1}^{\infty}\in\overset{\infty}{\underset{1}{\oplus}}\,\mathcal{D}_{\alpha}.
$$
 Similarly, we define the column operator $M_{F}^{C}:\mathcal{D}_{\alpha}\rightarrow\overset{\infty}{\underset{1}{\oplus}}\,\mathcal{D}_{\alpha}$\,
by
$$
M_{F}^{C}\left(h\right)=\left\{ f_{j}h\right\} _{j=1}^{\infty}\,\mbox{ for }\, h\in\mathcal{D}_{\alpha}.$$

\medskip{}
We notice that $\mathcal{D}_{\alpha}$ is a reproducing kernel (r.k.)
Hilbert space with r.k.
$$
K_{w}(z)=\underset{n=0}{\overset{\infty}{\sum}}\; \frac{1} {(n+1)^{\alpha}}\left(z\overline{w}\right)^{n}\,\mbox{ for }z,\, w\in\mathbb{D}$$

 and it is well known (see {[}S{]}) that
$$
\frac{1}{k_{w}(z)}=1-\overset{\infty}{\underset{n=1}{\sum}}\, c_{n}(z\overline{w})^{n},\, c_{n}>0,\mbox{ for all }n$$

 Hence, weighted Dirichlet space has a reproducing kernel with {}``one
positive square'' or a {}``complete Nevanlinna-Pick'' kernel. This
property will be used to complete the first part of our proof.

\medskip{}
We know that $\mathcal{M}(\mathcal{D}_{\alpha})\subseteq H^{\infty}(\mathbb{D})$,
but $\mathcal{M}(\mathcal{D}_{\alpha})\neq H^{\infty}(\mathbb{D})$
(e.g., $\overset{\infty}{\underset{n=1}{\sum}}\frac{z^{n^{4m+1}}}{n^{2m\alpha}},\, m=\left[\frac{1}{\alpha}\right]+1$, $z\in D,$
is in $H^{\infty}(D)$ but is not in $\mathcal{D_{\alpha}}$ and so
neither in $\mathcal{M}(\mathcal{D_{\alpha}})$). Hence, $\mathcal{M}(\mathcal{D_{\alpha}})\subsetneq H^{\infty}(\mathbb{D})\,\cap\,\mathcal{D_{\alpha}}$.

\medskip{}
Also, it is worthwhile to note that the point wise hypothesis that

\noindent
$F(z)\, F(z)^{\star}$ $\leq1$ for $z\in\mathbb{D}$ implies that
the analytic Toeplitz operators $T_{F}^{R}$ and $T_{F}^{C}$ defined
on $\overset{\infty}{\underset{1}{\oplus}}\, H^{2}(\mathbb{D})$ and
$H^{2}(\mathbb{D})$, in analogy to that of $M_{F}^{R}$ and $M_{F}^{C}$,
are bounded and
$$
\left\Vert T_{F}^{R}\right\Vert =\left\Vert T_{F}^{C}\right\Vert =\underset{z\in\mathbb{D}}{sup}\left(\overset{\infty}{\underset{j=1}{\sum}}\,\vert f_{j}(z)\vert^{2}\right)^{\frac{1}{2}}\le1.
$$

\noindent But, since $M(\mathcal{D}_{\alpha})\varsubsetneq H^{\infty}(\mathbb{D})$,
the point wise upperbound hypothesis will not be sufficient to conclude
that $M_{F}^{R}$ and $M_{F}^{C}$ are bounded on weighted Dirichlet space.
However, $\left\Vert M_{F}^{R}\right\Vert \leq\sqrt{10}\left\Vert M_{F}^{C}\right\Vert $ (see [KT]).
Thus, we will replace the natural normalization that $F(z)\, F(z)^{\star}\leq1$
for all $z\in\mathbb{D}$ by the stronger condition that $\left\Vert M_{F}^{C}\right\Vert \leq1.$

\bigskip{}
Then we have the following theorem: 
\begin{theorem} Let $H$,$\{f_{j}\}_{j=1}^{\infty}\subset\mathcal{M}(\mathcal{D}_{\alpha})$.
Assume that
\begin{align*}
 & (\mathrm{a})\;\Vert M_{F}^{C}\Vert\le1\\
\text{and }\;\; & (\mathrm{b})\;|H(z)|\le\sqrt{\sum_{j=1}^{\infty}\,|f_{j}(z)|^{2}}\;\text{ for all }z\in\mathbb{D}.
\end{align*}

Then there exists $K(\alpha)< \infty$ and there exists $\{g_{j}\}_{j=1}^{\infty}\subset\mathcal{M}(\mathcal{D}_{\alpha})$
with
\begin{align*}
 & \;\Vert M_{G}^{C}\Vert \leq K(\alpha)\\
\text{and }\;\; & \; F\, G^{T}=H^{3}.
\end{align*}
 \end{theorem}

\medskip{}
Of course, it should be noted that for only a finite number of multipliers,
$\{f_{j}\}$, condition (a) of Theorem 1 can always be assumed, so
we have the exact analogue of Wolff's theorem in the finite case.

\section{Outline of the proof of Theorem 1}
In this section, we will collect some known results and also prove required lemmas and then give an outline of the proof of Theorem 1.\\
 \indent  Assume that $F\in\mathcal{M}_{l^{2}}(\mathcal{D}_{\alpha})$
and $H\in\mathcal{M}(\mathcal{D}_{\alpha})$ satisfy the hypotheses
(a) and (b) of Theorem 1. Then we show that there exists a constant
$K(\alpha)<\infty,$ so that
\begin{equation}
M_{H^{3}}\, M_{H^{3}}^{\star}\leq K(\alpha)^{2}M_{F}^{R}\, M_{F}^{\star R}.
\end{equation}

Given $(2)$, a commutant lifting theorem argument as it appears in,
for example, Trent {[}Tr2{]} completes the proof by providing a $G\in\mathcal{M}_{l^{2}}(\mathcal{D}_{\alpha})$,
so that $\Vert M_{G}^{C}\Vert\leq K(\alpha)$ and $F\, G^{T}=H^{3}$.\\

But (2) is equivalent to the following: there exists a constant
 $K(\alpha)<\infty$
so that, for any $h\in\mathcal{D}_{\alpha}$, there exists \,$\underline{u}_{h}\in\overset{\infty}{\underset{1}{\oplus}}\,\mathcal{D}_{\alpha}$\,
such that
\begin{align}
(\mathrm{i}) & \,\,\, M_{F}^{R}(\underline{u}_{h})=H^{3}h\quad\text{ and}\notag\\
(\mathrm{ii}) & \,\,\left\Vert \underline{u}_{h}\right\Vert _{\mathcal{D_{\alpha}}}\leq K(\alpha)\left\Vert h\right\Vert _{\mathcal{D_{\alpha}}}.
\end{align}

\medskip{}
Hence, our goal is to show that (3) follows from (a) and (b). For
this we need a series of lemmas.
\begin{lemma}
Let $\left\{ c_{j}\right\} _{j=1}^{\infty}\in l^{2}$ and $C=\left(c_{1},c_{2},...\right)\in B\left(l^{2},\mathbb{C}\right).$
Then there exists $Q$ such that the entries of $Q$ are either $0$
or $\pm c_{j}$ for some $j$ and $CC^{\star}I-C^{\star}C=QQ^{\star}.$
Also, range of $Q$ $=$ kernel of $C.$
\end{lemma}
We will apply this lemma in our case with $C=F(z)$ for each $z\in\mathbb{D}$,
when $F(z)\neq0$. A proof of this lemma can be found in Trent {[}Tr2{]}. We can see in the proof that $Q(z)$ is analytic in $z$ on $\mathbb{D}$.

\medskip{}
Given condition (b) of Theorem 1 for all $z\in\mathbb{D}$, $F\in\mathcal{M}_{l^{2}}(\mathcal{D}_{\alpha})$
and $H\in\mathcal{M}(\mathcal{D}_{\alpha})$ with $H$ being not identically
zero, we lose no generality assuming that $H(0)\neq0.$ If $H(0)=0,$
but $H(a)\neq0,$ let $\beta(z)=\frac{a-z}{1-\bar{a}\, z}$ for $z\in\mathbb{D}$.
Then since (b) holds for all $z\in\mathbb{D}$, it holds for $\beta(z)$.
So we may replace $H$ and $F$ by $Ho\beta$ and $Fo\beta,$ respectively.
If we prove our theorem for $Ho\beta$ and $Fo\beta$, then there
exists $G\in\mathcal{M}_{l^{2}}(\mathcal{D}_{\alpha}$) so that $\left(Fo\beta\right)G=Ho\beta$
and hence $F(Go\beta^{-1})=H\,$ and $Go\beta^{-1}\in\mathcal{M}_{l^{2}}(\mathcal{D}_{\alpha}),$
and we are done. Thus, we may assume that $H(0)\neq0$ in (b), so
$\Vert F(0)\Vert_{2}\neq0.$ This normalization will let us apply
some relevant lemmas from {[}Tr1{]}.

\medskip{}
It suffices to establish (i) and (ii) for any dense set of functions
in $\mathcal{D}_{\alpha}$, so we will use polynomials. First, we
will assume $F$ and $H$ are analytic on $\mathbb{D}{}_{1+\epsilon}(0)$.
In this case, we write the most general solution of the pointwise
problem on $\overline{\mathbb{D}}$ and find an analytic solution
with uniform bounds. Then we remove the smoothness hypotheses on $F$
and $H$.

\medskip{}
For a polynomial, $h$, we take
$$
\underline{u}_{h}(z)=F(z)^{\star}\left(F(z)F(z)^{\star}\right)^{-1}H^{3}\, h-Q(z)\underline{k}(z),\mbox{where }\underline{k}(z)\in l^{2}\mbox{ for }z\in\mathbb{D}.$$

 We have to find $\underline{k}(z)$ so that $\underline{u}_{h}\in\overset{\infty}{\underset{1}{\oplus}}\,\mathcal{D}_{\alpha}$.
Thus we want $\bar{\partial_{z}}\,\underline{u}_{h}=0$ in $\mathbb{D}.$

\medskip{}
Therefore, we will try
$$
\underline{u}_{h}=\frac{F^{\star}H^{3}h}{FF^{\star}}-Q\,\,\widehat{W},
$$
 where $$W= {\left(\frac{Q^{\star}F^{'\star}H^{3}h}{\left(FF^{\star}\right)^{2}}\right)}$$  and  $\widehat{W}$ is the Cauchy transform of $W$ on $\mathbb{D}$.
Note that for $k$ smooth on $\overline{\mathbb{D}}$ and $z\in\mathbb{D}$,
$$
\underline{\widehat{k}}(z)=-\frac{1}{\pi}\int_{D}\frac{\underline{k}(w)}{w-z}\, dA(w)\,\text{ and }\;\overline{\partial}\,\underline{\widehat{k}}(z)=k(z)\,\text{ for }z\in\mathbb{D}.
$$
 See {[}A{]} for background on the Cauchy transform.

\medskip
Then it's clear that $M_{F}^{R}\left(\underline{u}_{h}\right)=H^{3}h$
and $\underline{u}_{h}$ is analytic. Hence, we will be done in the
smooth case if we are able to find $K(\alpha)<\infty$, only depending
on $\alpha$ and thus independent of the polynomial, $h$, such that
\begin{equation}
\left\Vert \underline{u}_{h}\right\Vert _{\mathcal{D}_{\alpha}}\leq K(\alpha)\left\Vert h\right\Vert _{\mathcal{D}_{\alpha}}
\end{equation}




\begin{lemma}
Let the operator T be defined on $L^{2}\left(\mathbb{D}, dA_{\alpha}\right)$
by
$$
\left(Tf\right)(z)=\int_{D}\mbox{\ensuremath{\frac{f(u)}{(u-z)(1-u\bar{z})}}}dA_{\alpha},\,
$$
 for $z \in\mathbb{D}$ and $f\in L^{2}(\mathbb{D}, dA_{\alpha})$.
Then
$$
\vert\vert Tf\vert\vert_{A_{\alpha}}^{2}  \leq4\pi^{2}C_{\alpha}^{2}\vert\vert f\vert\vert_{A_{\alpha}}^{2},
$$
where $ C_{\alpha}=\frac{8}{{\alpha}^2}$.
\end{lemma}

\begin{proof}
To show that the singular integral operator, $T$, is bounded on $L^{2}(\mathbb{D},dA_{\alpha})$,
we apply Zygmund's method of rotations {[}Z{]} and apply Schur's lemma an
infinite number of times.

\medskip{}
Let $f(z)=\overset{\infty}{\underset{j=0}{\sum}}\,\overset{\infty}{\underset{k=0}{\sum}}\, a_{jk}z^{j}\bar{z}^{k}$,
where $a_{ij}=0$ except for a finite number of terms. For $z=r\, e^{i\theta}$, we relabel to get
$$
f(r\, e^{i\theta})=\overset{\infty}{\underset{l=-\infty}{\sum}}f_{l}(r)\, e^{il\theta},\text{ where }\, f_{l}(r)=\overset{\infty}{\underset{k=0}{\sum}}a_{l+k\, k\,}r^{l+2k}.
$$
 Then
$$
\Vert f\Vert_{A_{\alpha}}^{2}=\overset{\infty}{\underset{l=-\infty}{\sum}}\Vert f_{l}(r)\Vert_{L_{\alpha}^{2}\left[0,1\right]}^2,
$$
 where the measure on $L_{\alpha}^{2}\left[0,1\right]$ is {}``$\left(1-r^{2}\right)^{1-\alpha}rdr$''.

\medskip{}
Now computing as in [BT], we deduce that
$$
\left(Tf\right)\left(se^{it}\right)=2\pi\underset{l=-\infty}{\overset{\infty}{\sum}}e^{i(l-1)t}\left(T_{l}f_{l}\right)(s),
$$

$$
\mbox{for \; \ensuremath{\left(T_{l}f_{l}\right)(s)=\begin{cases}
\begin{array}{l}
-(\underset{n=0}{\overset{-l}{\sum}}s^{2n})\int_{0}^{1}\chi_{\left(0,s\right)}(r)\left(\frac{r}{s}\right)^{1-l}f_{l}(r)\, dr\\
\quad+\frac{1}{1-s^{2}}\int_{0}^{1}\chi_{(s,1)}(r)\,\left(rs\right)^{1-l}f_{l}(r)\, dr\quad\text{for }\, l\leq0
\end{array}\\
\;\frac{1}{1-s^{2}}\int_{0}^{1}\chi_{(s,1)}(r)\left(\frac{s}{r}\right)^{l-1}f_{0}(r)\, rdr\quad\quad\;\;\text{for }\, l>0.
\end{cases}}}
$$

\medskip{}
By our construction,
$$
\Vert Tf\Vert_{A_{\alpha}}^{2}=4\,\pi^{2}\underset{l=-\infty}{\overset{\infty}{\sum}}\vert\vert T_{l}f_{l}\vert\vert_{L_{\alpha}^{2}\left[0,1\right]}^{2},
$$
 where the measure on $L^{2}\left[0,1\right]$ is {}``$\left(1-r^{2}\right)^{1-\alpha}rdr$''.
Thus, to prove our lemma it suffices to prove that
\begin{equation}
\underset{l}{sup}\,\Vert T_{l}\Vert_{B\left(L_{\alpha}^2\left[ 0,1 \right]\right)}\, \leq \, C_{\alpha}<\infty.\notag
\end{equation}

\medskip{}
\noindent
To illustrate the technique, we show a detailed estimate for
$\Vert T_{0}\Vert_{B\left(L_{\alpha}^2\left[ 0,1 \right]\right)}$. The other cases follow similarly.

\medskip
Now

\begin{align*}
 & \int_{0}^{1}\left|T_{0}f_{0}(se^{it})\right|^{2}(1-s^{2})^{1-\alpha}sds\\
  & =2\,\int_{0}^{1}\int_{0}^{1}f_{0}(u)f_{0}(v)\left(\int_{max\left\{ u,v\right\} }^{1}\frac{(1-s^{2})^{1-\alpha}ds}{s}\right)udu\, vdv\\
 & \quad+2\int_{0}^{1}\int_{0}^{1}\, f_{0}(x)\, f_{0}(y)\left[\int_{0}^{min\left\{ x,\, y\right\} }\,\mbox{\ensuremath{\frac{s^{2}\left(1-s^{2}\right)^{1-\alpha}}{(1-s^{2})^{2}}sds}}\right]xdx\, ydy.
\end{align*}

\medskip{} \noindent
\textbf{Claim (I):}

\begin{align*}
&\int_{0}^{1}\int_{0}^{1}f_{0}(u)f_{0}(v)\left(\int_{max\left\{ u,v\right\} }^{1}\frac{(1-s^{2})^{1-\alpha}ds}{s}\right)udu\, vdv\\
& \qquad \qquad \leq\frac{25}{16}\int_{0}^{1}\left|f_{0}(u)\right|^{2}\left(1-u^{2}\right)^{1-\alpha}u\, du.
\end{align*}

\medskip{}
We have
\begin{align*}
 & \int_{0}^{1}\int_{0}^{1}f_{0}(u)f_{0}(v)\left(\int_{max\left\{ u,v\right\} }^{1}\frac{(1-s^{2})^{1-\alpha}ds}{s}\right)udu\, vdv\\
  & \quad \leq \int_{0}^{1}\int_{0}^{1}f_{0}(u)\, f_{0}(v)\left[\frac{\left(1-max(u^{2},v^{2})\right)^{1-\alpha}}{\left(1-u^{2}\right)^{1-\alpha}\left(1-v^{2}\right)^{1-\alpha}}ln\left(\frac{1}{max\left\{ u,v\right\} }\right)\right]\\
  & \qquad \qquad \qquad \qquad \quad \qquad \left(1-u^{2}\right)^{1-\alpha}\left(1-v^{2}\right)^{1-\alpha}udu\, vdv.
\end{align*}

We apply Schur's Test with $p(u)=1$.

\begin{align*}
&\int_{0}^{v}\left[\frac{\left(1-v^{2}\right)^{1-\alpha}}{\left(1-u^{2}\right)^{1-\alpha}\left(1-v^{2}\right)^{1-\alpha}}ln\left(\frac{1}{v}\right)\right]\left(1-u^{2}\right)^{1-\alpha}udu\\
 & \qquad \qquad =\frac{1}{2}\, ln\left(\frac{1}{v^{2}}\right)\frac{v^{2}}{2}
  \leq\frac{1}{4}.
\end{align*}
\medskip{}
Similarly, we get $\int_{v}^{1}\left[\frac{\left(1-u^{2}\right)^{1-\alpha}}{\left(1-u^{2}\right)^{1-\alpha}\left(1-v^{2}\right)^{1-\alpha}}ln\left(\frac{1}{u}\right)\right]\left(1-u^{2}\right)^{1-\alpha}u du\leq1$.

\medskip{}
Therefore,
\begin{align*}
& \int_{0}^{1}\left[\frac{\left(1-max(u^{2},v^{2})\right)^{1-\alpha}}{\left(1-u^{2}\right)^{1-\alpha}\left(1-v^{2}\right)^{1-\alpha}}ln\left(\frac{1}{max\left\{ u,v\right\} }\right)\right]p(u)\left(1-u^{2}\right)^{1-\alpha}udu \\
& \qquad \qquad \leq \frac{5}{4}\, p(v).
\end{align*}

\medskip \noindent
\textbf{Claim (II):}\\
$$
\begin{aligned} & \int_{0}^{1}\int_{0}^{1}\, f_{0}(x)\, f_{0}(y)\left[\int_{0}^{min\left\{ x,\, y\right\} }\,\mbox{\ensuremath{\frac{s^{2}\left(1-s^{2}\right)^{1-\alpha}}{(1-s^{2})^{2}}sds}}\right]xdx\, ydy\\
 & \quad \leq \frac{4}{\alpha^2} \int_{0}^{1}\left|f_{0}(x)\right|^{2}\left(1-x^{2}\right)^{1-\alpha}xdx.
\end{aligned}
$$

We have
\begin{align*}
 & \int_{0}^{1}\int_{0}^{1}\, f_{0}(x)\, f_{0}(y)\left[\int_{0}^{min\left\{ x,\, y\right\} }\,\mbox{\ensuremath{\frac{s^{2}\left(1-s^{2}\right)^{1-\alpha}}{(1-s^{2})^{2}}sds}}\right]xdx\, ydy\\
 & =\int_{0}^{1}\int_{0}^{1}\, f_{0}(x)\, f_{0}(y)\left[\frac{1}{2}\int_{0}^{min\left\{ x^{2},y^{2}\right\} }\frac{s}{(1-s)^{1+\alpha}}ds\right]xdx\, ydy\\
 & \leq\int_{0}^{1}\int_{0}^{1}\, f_{0}(x)\, f_{0}(y)\left[\frac{1}{2\alpha}\frac{min\left\{ x^{2},y^{2}\right\} }{\left(1-min\left\{ x^{2},y^{2}\right\} \right)^{\alpha}}\right]xdx\, ydy.
\end{align*}

\medskip{}

For this term, we take $p(x)=\frac{1}{\left(1-x^{2}\right)^{\beta}}$,
 where $\beta=1-\frac{\alpha}{2}$.
Then, calculating, we get that

\[
\int_{0}^{y}\,\frac{1}{2\alpha}\,\frac{x^{2}}{\left(1-x^{2}\right)^{\alpha+\beta}}\,\frac{1}{\left(1-y^{2}\right)^{1-\alpha}}\, xdx\leq\frac{1}{4\alpha\left(\beta+\alpha-1\right)}\frac{1}{\left(1-y^{2}\right)^{\beta}}.
\]

\medskip{}
Similarly,

\[
\int_{y}^{1}\,\frac{1}{2\alpha}\,\frac{y^{2}}{\left(1-y^{2}\right)^{\alpha}}\,\frac{1}{\left(1-y^{2}\right)^{1-\alpha}}\frac{1}{\left(1-x^{2}\right)^{\beta}}\, xdx\leq\frac{1}{4\alpha\left(\beta-1\right)}\frac{1}{\left(1-y^{2}\right)^{\beta}}.
\]

\medskip{}

Therefore,
\begin{align*}
 & \int_{0}^{1}\left[\frac{1}{2\alpha}\frac{min\left\{x^{2},y^{2}\right\}}{\left(1-min\left\{x^{2},y^{2}\right\}\right)^{\alpha}\left(1-x^{2}\right)^{1-\alpha}\left(1-y^{2}\right)^{1-\alpha}}\right] p(x)\left(1-x^{2}\right)^{1-\alpha}xdx\\
 & \qquad =\left(\frac{1}{4\alpha\left(\beta+\alpha-1\right)}+\frac{1}{4\alpha\left(1-\beta\right)}\right)p(y)\\
 & \qquad =\frac{1}{\left(4\beta+\alpha-1\right)\left(1-\beta\right)}\, p(y)=\frac{1}{\alpha^2}\, p(y).
\end{align*}

\medskip{}

Hence,
\[
\int_{0}^{1}\left|T_{0}f_{0}(s)\right|^{2}(1-s^{2})^{1-\alpha}sds  \leq C_{\alpha_{0}}^{2}\int_{0}^{1}\left|f_{0}(s)\right|^{2}\left(1-s^{2}\right)^{1-\alpha}sds,
\]
where $C_{\alpha_{0}}=\left[\frac{5}{2}+\frac{2}{\alpha^2}\right]\leq \frac{5}{\alpha^2}$.

\medskip{}

Applying Schur's test for $l>1$ with $p(x)=\frac{1}{\left(1-x^{2}\right)^{\beta}},\,\beta=1-\frac{\alpha}{2}$,
we get the estimate $C_{l}\leq\frac{5}{\alpha^2}$,
independent of $l$. Similarly, for $l<0$ with $p(x)=1$ and $p(x)=\frac{1}{\left(1-x^{2}\right)^{\beta}},\,$
for each of the two terms, respectively, we get the estimate $C_{l}\leq 6+\frac{2}{\alpha^2}$,
independent of $l$ . Thus we conclude that
\[
\underset{l}{\sup}\,\Vert T_{l}\Vert_{B\left(L_{\alpha}^{2}\left[0,1\right]\right)}\leq \frac{8}{\alpha^2}.
\]
This finishes the proof of the Lemma. \end{proof}

\medskip

A classical treatment of the Beurling transform can be found in Zygmund [Z]. For our purposes, we define the Beurling transform formally by

\[
\mathcal{B}(\phi)=\partial_{z}(\widehat{\phi} ),
\] 
where $\phi$ is in $C^1(\overline{\mathbb{D}})$
and $\widehat{\phi}$ is the Cauchy transform of $\phi$ on $\mathbb{D}$.

\medskip
\begin{lemma}
Let~$\mathcal{B}$ denote the Beurling transform. Then

\[
\vert\vert\mathcal{B}\left(f\right)\vert\vert_{A_{\alpha}}\leq \frac{23}{\alpha}\,\vert\vert f\vert\vert_{A_{\alpha}},\,\,\mbox f\in L^{2}\left(\mathrm{\mathbb{D},\: dA_{\alpha}}\right).
\]
\end{lemma}

\begin{proof}
To show that the Beurling transform, $\mathcal{B}$, is bounded on
$L^{2}\left(\mathbb{D,\: dA_{\alpha}}\right),$ we again apply Zygmund's
method of rotations {[}Z{]} and apply Schur's lemma.

As in Lemma 3, we take
\[
f(r\, e^{i\theta})=\overset{\infty}{\underset{l=-\infty}{\sum}}f_{l}(r)\, e^{il\theta},\text{ where }\, f_{l}(r)=\overset{\infty}{\underset{k=0}{\sum}}a_{l+k\, k\,}r^{l+2k}.
\]
 Then
\[
\Vert f\Vert_{A_{\alpha}}^{2}=\overset{\infty}{\underset{l=-\infty}{\sum}}\Vert f_{l}(r)\Vert_{L_{\alpha}^{2}\left[0,1\right]}^2,
\]
 where the measure on $L_{\alpha}^{2}\left[0,1\right]$ is {}``$\left(1-r^{2}\right)^{1-\alpha}rdr$''.

\medskip{}

Now
\begin{align*}
\overset{}{\widehat{f}(w)} & \overset{}{=-\frac{2}{2\pi}}\int_{D}\frac{f(z)}{z-w}\, dA(z)\\
  & =2\overset{\infty}{\underset{l=-\infty}{\sum}}\overset{\infty}{\underset{n=0}{\sum}}\int_{0}^{2\pi}\int_{0}^{\vert w\vert}\,\,\frac{f_{l}(r)\,\, e^{i(l+n)\,\theta}}{w^{n+1}}\, r^{n+1}drd\sigma(\theta)\\
 & \quad-2\overset{\infty}{\underset{l=-\infty}{\sum}}\overset{\infty}{\underset{n=0}{\sum}}\int_{0}^{2\pi}\int_{\vert w\vert}^{1}\,\frac{f_{l}(r)\, e^{i\,(l-1-n)\,\theta}w^{n}}{r^{n}}drd\sigma(\theta).\,\,\,\,\,\,\,\,\,\,\, & (\star)
\end{align*}

\medskip{}
If we take $l=0$ in $(\star)$, we get that

\[
\widehat{f_{0}}(w)=\frac{2}{w}\int_{0}^{\vert w\vert}\, f_{0}(r)\, rdr.
\]

Therefore,
\begin{align*}
\partial\widehat{f_{0}}(w) & =\frac{-2}{w^{2}}\int_{0}^{\vert w\vert}\, f_{0}(r)\, rdr+\frac{2}{w}\, f_{0}\left(\vert w\vert\right)\vert w\vert\frac{\partial(\vert w\vert)}{\partial w}\\
 & =\frac{-2}{w^{2}}\int_{0}^{\vert w\vert}\, f_{0}(r)\, rdr+\frac{\overline{w}}{w}\, f_{0}\left(\vert w\vert\right),
 \end{align*}
 since\,\,$\overline{w}=\frac{\partial\vert w\vert^{2}}{\partial w}=2\vert w\vert\frac{\partial\vert w\vert}{\partial w},\,\,\frac{\partial\vert w\vert}{\partial w}=\frac{\overline{w}}{2\vert w\vert}$.

\medskip{}
Thus,

\[
\mathcal{B}f_{0}(se^{it})=\partial\widehat{f_{0}}(se^{it})=e^{-2it}\left[\frac{-2}{s^{2}}\int_{0}^{s}\, f_{0}(r)\, rdr+f_{0}\left(s\right)\right].
\]

\medskip{}

Similarly, a computation shows that
\[
\mathcal{B}(f)(se^{it})=\overset{\infty}{\underset{l=-\infty}{\sum}}e^{i(l-2)t}\mathcal{B}_{l}f_{l}(s),
\]

\begin{align*}
\mbox{for }\;\mathcal{B}_{l}f_{l}(s) & =\begin{cases}
\begin{aligned}\frac{-2}{s^{2}}\int_{0}^{s}\, f_{0}(r)\, rdr+f_{0}\left(s\right) & \mbox{\,\,\,\ for \ensuremath{l=0}}\\
\\
-2\,(l-1)s^{l-2}\int_{s}^{1}\,\frac{f_{l}(r)}{r^{l-1}}dr-f_{l}(s) & \mbox{\,\,\,\ for \ensuremath{l\geq1}}\\
\\
\mbox{\ensuremath{-2(1-l)s^{l-2}\int_{0}^{s}f_{l}(r)r^{1-l}dr+f_{l}(s})} & \mbox{\,\,\,\ for \ensuremath{l<0}}.
\end{aligned}
\end{cases}
\end{align*}

\medskip{}
Thus,
\[
\vert\vert\mathcal{B}f\vert\vert_{A_{\alpha}}^{2}=\overset{\infty}{\underset{l=-\infty}{\sum}}\vert\vert\mathcal{B}_{l}f_{l}\vert\vert_{L_{\alpha}^{2}\left[0,1\right],}^{2}
\]
where the measure on $L_{\alpha}^{2}\left[0,1\right]$ is {}``$\left(1-r^{2}\right)^{1-\alpha}rdr$''.

\bigskip{}\noindent
\textbf{Claim:}
\[
\underset{l}{\sup}\vert\vert\mathcal{B}_{l}\vert\vert_{B\left(L_{\alpha}^{2}\left[0,1\right]\right)}\leq \frac{23}{\alpha}<\infty.
\]

\smallskip{}
Without loss of generality we may assume that $f_{l}(s)\geq0\,\,\mbox{for all }\, l.$
For $l<2$, applying Schur's test with $p(u)=1$ or $p(u)=\frac{1}{\sqrt{u}}$, we get that $\vert\vert\mathcal{B}_{l}\vert\vert_{B\left(L^{2}\left[0,1\right]\right)}\leq 7$. The main cases occur for $l\geq 2$. So let $l \geq 2$ be fixed. Then

\begin{align*}
\vert\vert\mathcal{B}_{l}f_{l}\vert\vert_{L_{\alpha}^{2}\left[0,1\right]} & \leq 2 \left(\int_{0}^{1}\vert-\,(l-1)s^{l-2}\int_{s}^{1}\,\frac{f_{l}(r)}{r^{l-1}}dr\vert^{2}(1-s^{2})^{1-\alpha}sds \right)^{\frac{1}{2}}\\
& \qquad \qquad +\vert\vert f_{l}\vert\vert_{L_{\alpha}^{2}\left[0,1\right]}
\end{align*}
Now,
\begin{align*}
& (l-1)^{2}\int_{0}^{1}s^{2(l-2)}\vert\int_{0}^{1}\chi_{(s,1)}(r)\,\frac{f_{l}(r)}{r^{l-1}}dr\vert^{2}(1-s^{2})^{1-\alpha}sds\\
& \quad =\int_{0}^{1}\int_{0}^{1}f_{l}(u)f_{l}(v)\left[\left(l-1\right)^{2}\frac{1}{u^{l}}\,\frac{1}{v^{l}}\frac{\int_{0}^{min\left\{u,v\right\}}s^{2(l-2)}(1-s^{2})^{1-\alpha}sds}
 {(1-u^{2})^{1-\alpha}(1-v^{2})^{1-\alpha}}\right]\\
& \qquad \qquad \qquad\qquad \quad (1-u^{2})^{1-\alpha}(1-v^{2})^{1-\alpha}udu\, vdv.
\end{align*}

\medskip{}

Applying Schur's test with $p(u)=\frac{1}{\left(1-u^{2}\right)^{1-\alpha}}$, then
it's sufficient to show that

\[
\int_{0}^{1}\left[\left(l-1\right)^{2}\frac{1}{u^{l}}\,\frac{\int_{0}^{min\left\{ u,v\right\} }s^{2l-3}(1-s^{2})^{1-\alpha}ds}{(1-u^{2})^{1-\alpha}}\right]udu\leq C_{l}\, v^{l}.
\]

\medskip{}
Since $(1+s)^{1-\alpha}\leq2$
and $\frac{1}{2}\leq\frac{1}{(1+u)^{1-\alpha}}\leq1$,
 we will be done if we are able to show

\[
\int_{0}^{1}\left[\left(l-1\right)^{2}\frac{1}{u^{l}}\,\frac{\int_{0}^{min\left\{ u,v\right\} }s^{2l-3}(1-s)^{1-\alpha}ds}{(1-u)^{1-\alpha}}\right]udu\leq C_{l}\, v^{l}.
\]

So we are trying to prove that
\begin{align*} 
 \int_{0}^{v}\left[\left(l-1\right)^{2}\frac{1}{u^{l}}\,\frac{\int_{0}^{u}s^{2l-3}(1-s)^{1-\alpha}ds}{(1-u)^{1-\alpha}}\right]udu & \leq C_{l}v^{l}\,\,\,\,\,\mbox{and }\\
 \int_{v}^{1}\left[\left(l-1\right)^{2}\frac{1}{u^{l}}\,\frac{\int_{0}^{v}s^{2l-3}(1-s)^{1-\alpha}ds}{(1-u)^{1-\alpha}}\right]udu & \leq C_{l}v^{l}.
\end{align*}

Now
\begin{align*}
 & \int_{0}^{v}\left[\left(l-1\right)^{2}\frac{1}{u^{l}}\,\frac{\int_{0}^{u}s^{2l-3}(1-s)^{1-\alpha}ds}{(1-u)^{1-\alpha}}\right]udu\\
  & \; =\int_{0}^{v}\left[\left(l-1\right)^{2}s^{2l-3}(1-s)^{1-\alpha}\int_{s}^{v}\frac{du}{u^{l-1}\,(1-u)^{1-\alpha}}\right]ds.
\end{align*}

Let $t=(1-u)^{\alpha}$ and change variables.
Then we get that 
\begin{align*}
&\int_{0}^{v}\left[\left(l-1\right)^{2}s^{2l-3}(1-s)^{1-\alpha}\int_{s}^{v}\frac{du}{u^{l-1}(1-u)^{1-\alpha}}\right]ds\\
 & =\int_{0}^{v}\left[\frac{1}{\alpha}\left(l-1\right)^{2}s^{2l-3}(1-s)^{1-\alpha}\int_{(1-v)^{\alpha}}^{(1-s)^{\alpha}}\frac{dt}{\left(1-t^{\frac{1}{\alpha}}\right)^{(l-2)+1}}\right]ds\\
 & =\int_{0}^{v}\left[\frac{1}{\alpha}\left(l-1\right)^{2}s^{2l-3}(1-s)^{1-\alpha}\underset{p=0}{\overset{\infty}{\sum}}\left(\begin{array}{c}
l-2+p\\
p
\end{array}\right)\int_{(1-v)^{\alpha}}^{(1-s)^{\alpha}}t^{\frac{p}{\alpha}}dt\right]ds\\
  & \leq\int_{0}^{v}\left[\frac{1}{\alpha}\left(l-1\right)^{2}s^{2l-3}(1-s)^{1-\alpha}\underset{p=0}{\overset{\infty}{\sum}}\frac{\left(l-2+p\right)!}{\left(l-2\right)!\, p!}\bigg[\frac{\left((1-s)^{\alpha}\right)^{\frac{p}{\alpha}+1}}{\frac{p}{\alpha}+1}\bigg]\right]ds\\
 & \leq\frac{2}{\alpha}\int_{0}^{v}\left[\left(l-1\right)s^{2l-3}(1-s)^{1-\alpha}\underset{q=1}{\overset{\infty}{\sum}}\frac{\left(l-3+q\right)!}{(l-3)!\, q!}\,\frac{(1-s)^{q}}{(1-s)^{1-\alpha}}\right]ds\\
 & = \frac{2}{\alpha}\int_{0}^{v}\left[\left(l-1\right)s^{2l-3}\left(\frac{1}{\left(1-\left(1-s\right)\right)^{l-3+1}}-1\right)\right]ds\\
 & \leq \frac{2}{\alpha}\int_{0}^{v}\left[\left(l-1\right)s^{2l-3}\left(\frac{1}{s^{l-2}}\right)\right]ds\\
 & \leq\frac{2}{\alpha}v^{l}.
\end{align*}

\medskip{}
Now consider

\begin{align*}
 & \int_{v}^{1}\left[\left(l-1\right)^{2}\frac{1}{u^{l}}\,\frac{\int_{0}^{v}s^{2l-3}(1-s)^{1-\alpha}ds}{(1-u)^{1-\alpha}}\right]udu\\
 & \; =\int_{0}^{v}\left[\left(l-1\right)^{2}s^{2l-3}(1-s)^{1-\alpha}\int_{v}^{1}\frac{du}{u^{l-1}\,(1-u)^{1-\alpha}}\right]ds.
\end{align*}

Again, change variables with $t=(1-u)^{\alpha}$.
So 
{\allowdisplaybreaks
\begin{align*}
& \int_{0}^{v}\left[\left(l-1\right)^{2}s^{2l-3}(1-s)^{1-\alpha}\int_{v}^{1}\frac{du}{u^{l-1}\,(1-u)^{1-\alpha}}\right]ds\\
 &  =\int_{0}^{v}\left[\frac{1}{\alpha}\left(l-1\right)^{2}s^{2l-3}(1-s)^{1-\alpha}\int_{0}^{(1-v)^{\alpha}}\frac{dt}{\left(1-t^{\frac{1}{\alpha}}\right)^{l-1}}\right]ds\\
  &  =\int_{0}^{v}\left[\frac{1}{\alpha}\left(l-1\right)^{2}s^{2l-3}(1-s)^{1-\alpha}\underset{p=0}{\overset{\infty}{\sum}}\left(\begin{array}{c}
l-2+p\\
p
\end{array}\right)\int_{0}^{(1-v)^{\alpha}}t^{\frac{p}{\alpha}}dt\right]ds\\
  & =\int_{0}^{v}\left[\frac{1}{\alpha}\left(l-1\right)^{2}s^{2l-3}(1-s)^{1-\alpha}\underset{p=0}{\overset{\infty}{\sum}}\frac{\left(l-3+p+1\right)!}{\left(l-2\right)(l-3)!\, p!}\left[\frac{(1-v)^{p+\alpha}}{p+1}\right]\right]ds\\
 & \leq\frac{2}{\alpha}\int_{0}^{v}\left[\left(l-1\right)s^{2l-3}(1-s)^{1-\alpha}\underset{q=1}{\overset{\infty}{\sum}}\frac{\left(l-3+q\right)!}{(l-3)!\, q!}\,\frac{(1-v)^{q}}{(1-v)^{1-\alpha}}\right]ds\\
 & = \frac{2}{\alpha}\int_{0}^{v}\left[\left(l-1\right)s^{2l-3}(1-s)^{1-\alpha}\big(\frac{1}{\left(1-\left(1-v\right)\right)^{l-3+1}}-1\big)\frac{1}{(1-v)^{1-\alpha}}\right]ds\\
 & \leq \frac{2}{\alpha}\int_{0}^{v}\left[\left(l-1\right)s^{2l-3}(1-s)^{1-\alpha}\left(\frac{1-v^{l-2}}{v^{l-2}}\right)\frac{(1-v)^{\alpha}}{(1-v)}\right]ds\\
 & \leq\frac{2}{\alpha}\int_{0}^{v}\left[\left(l-1\right)s^{2l-3}(1-s)^{1-\alpha}(1-s)^{\alpha}\left(\frac{1-v^{l-2}}{1-v}\right)\frac{1}{v^{l-2}}\right]ds\\
 & =\frac{2(l-1)}{\alpha}\int_{0}^{v}\left[\left(s^{2l-3}-s^{2l-2}\right)\left(\frac{1-v^{l-2}}{1-v}\right)\frac{1}{v^{l-2}}\right]ds\\
 &  =\frac{2(l-1)}{\alpha}\left[\left(\frac{v^{2l-2}}{2l-2}-\frac{v^{2l-1}}{2l-1}\right)\left(\frac{1-v^{l-2}}{1-v}\right)\frac{1}{v^{l-2}}\right]\\
  & =\frac{2(l-1)v^{l}}{\alpha}\left[\left(\frac{(1-v)}{2l-2}+v\left(\frac{1}{2l-2}-\frac{1}{2l-1}\right)\right)\left(\frac{1-v^{l-2}}{1-v}\right)\right]\\
  & \leq\frac{1}{\alpha}\,\, v^{l}+\frac{2(l-1)v^{l+1}}{\alpha}\left[\left(\frac{1}{2(l-1)(2l-1)}\right)\left(\frac{1-v^{l-2}}{1-v}\right)\right]\\
 & =\frac{1}{\alpha}\,\, v^{l}+\frac{v^{l+1}}{\alpha}\left[\frac{1}{2l-1}\left(\frac{1-v^{l-2}}{1-v}\right)\right]\\
 & \leq\frac{1}{\alpha}\,\, v^{l}+\frac{v^{l+1}}{\alpha}\frac{(l-2)}{(2l-1)}\\
 & \leq\frac{2}{\alpha}\,\, v^{l}.
\end{align*}}

Therefore,
\[
\int_{0}^{1}\left[\left(l-1\right)^{2}\frac{1}{u^{l}}\,\frac{1}{v^{l}}\frac{\int_{0}^{min\left\{ u^{2},v^{2}\right\} }s^{(l-2)}(1-s)^{1-\alpha}ds}{(1-u^{2})^{1-\alpha}(1-v^{2})^{1-\alpha}}\right]p(u)(1-u^{2})^{1-\alpha}udu
\]
\[
\leq \frac{4}{\alpha} \,p(v).
\]

\medskip{}
We conclude that
\[\underset{l}{sup} \, \vert\vert\mathcal{B}_{l}\vert\vert_{B\left(L_{\alpha}^{2}\left[0,1\right]\right)}\leq 15+ \frac{8}{\alpha}\leq \frac{23}{\alpha}.
\]
\vskip-1em
\end{proof}
\medskip
\begin{lemma}
If \,$Q$ is a multiplier of $\mathcal{D}_{\alpha}$, then
\[
\left(1-\vert z\vert^{2}\right)\vert\, Q'(z)\vert\leq\Vert M_{Q}\Vert_{B(\mathcal{D}_{\alpha})} \text{ for all } z \in \mathbb{D}.
\]
\end{lemma}
\begin{proof}
Define\; $\varphi:D\rightarrow D$ \, as \, $\varphi(z)=\frac{Q(z)}{\Vert M_{Q}\Vert_{B(\mathcal{D}_{\alpha})}}$
for all\, $z\in\mathbb{D}$. Now use the Schwarz lemma and the fact
that 
$\Vert\varphi\Vert_{\infty,\mathbb{D}}\leq\Vert M_{\varphi}\Vert_{\mathcal{B}(\mathcal{D}_{\alpha})}$
to complete the proof.
\end{proof}

\smallskip
\begin{lemma}
If $H\in\mathcal{M}\left(\mathcal{D}_{\alpha}\right)$, then
$
\vert H\,'\vert^{2}dA_{\alpha}$\,is a $\mathcal{D}_{\alpha}$-Carleson measure with the constant \, $4\vert\vert M_{H}\vert\vert_{B\left(\mathcal{D}_{\alpha}\right)}^{2}$.

\end{lemma}
\begin{proof}
To prove the lemma, we need to show that

\[
\int_{\mathbb{D}}\vert H\,'\vert^{2}\vert g\vert^{2}dA_{\alpha}\leq4\vert\vert M_{H}\vert\vert_{B\left(\mathcal{D}_{\alpha}\right)}^{2}\vert\vert g\vert\vert_{\mathcal{D}_{\alpha}}^{2}\,\,\,\mbox{for all}\,\, g\in\mathcal{D}_{\alpha}.
\]

\medskip{}

Let $g\in\mathcal{D}_{\alpha}$, then

\begin{align*}
\int_{\mathbb{D}}\vert H\,'\vert^{2}\vert g\vert^{2}dA_{\alpha} & =\int_{\mathbb{D}}\vert\left(Hg\right)'-Hg'\vert^{2}dA_{\alpha}\\
 & \leq2\int_{\mathbb{D}}\vert\left(Hg\right)'\vert^{2}dA_{\alpha}+2\int_{\mathbb{D}}\vert Hg'\vert^{2}dA_{\alpha}\\
 & \leq2\int_{\mathbb{D}}\vert Hg\vert^{2}d\sigma+2\int_{\mathbb{D}}\vert\left(Hg\right)'\vert^{2}dA_{\alpha}+2\int_{\mathbb{D}}\vert Hg'\vert^{2}dA_{\alpha}\\
 & \leq2\vert\vert M_{H}g\vert\vert_{\mathcal{D}_{\alpha}}^{2}+2\vert\vert M_{H}\vert\vert_{B\left(\mathcal{D}_{\alpha}\right)}^{2}\vert\vert g\vert\vert_{\mathcal{D}_{\alpha}}^{2}\\
 & \leq4\vert\vert M_{H}\vert\vert_{B\left(\mathcal{D}_{\alpha}\right)}^{2}\vert\vert g\vert\vert_{\mathcal{D}_{\alpha}}^{2}.
\end{align*}
\vskip-1em
This proves the Lemma.
\end{proof}

\section{Proof of Theorem 1}

 First, we will prove the theorem for smooth functions
on $\overline{\mathbb{D}}$ and get a uniform bound. Then we will
 use a compactness argument to remove the smoothness hypothesis.

\medskip{}
Assume that (a) and (b) of Theorem 1 hold for $F$ and $H$ and that
$F$ and $H$ are analytic on $\mathbb{D}_{1+\epsilon}(0)$.
Our main goal is to show that there exists a constant, $K(\alpha)<\infty$,
independent of $\epsilon$, so that for any polynomial, $h$, there
exists $\underline{u}_{h}\in\overset{\infty}{\underset{1}{\oplus}}\,\mathcal{D}_{\alpha}$
such that $M_{F}^{R}(\underline{u}_{h})=H^{3}h$ and $\Vert\underline{u}_{h}\Vert_{\mathcal{D}_{\alpha}}^{2}\leq K(\alpha)\,\Vert h\Vert_{\mathcal{D}_{\alpha}}^{2}$.

We take \;  $\underline{u}_{h}=\frac{F^{\star}H^{3}h}{FF^{\star}}-Q\widehat{W},$ where  $W= \frac{Q^{\star}F^{'\star}H^{3}h}{\left(FF^{\star}\right)^{2}}$.
Then $\underline{u}_{h}$ is analytic and $M_{F}^{R}(\underline{u}_{h})=H^{3}h.$
We know that
\[
\left\Vert \underline{u}_{h}\right\Vert _{\mathcal{D}_{\alpha}}^{2}=\int_{-\pi}^{\pi}\Vert\underline{u}_{h}(e^{it})\Vert^{2}\, d\sigma(t)+\int_{D}\Vert\left(\underline{u}_{h}(z)\right)^{'}\Vert^{2}\, dA_{\alpha}(z).
\]
 Condition (b) implies that
\[
\int_{-\pi}^{\pi}\left\Vert\frac{F^{\star}H^{3}h}{FF^{\star}}-Q\,\widehat{W}\right\Vert^{2}\, d\sigma(t)\leq15\left\Vert h\right\Vert_{\sigma}^{2}\mbox{ (see [Tr1]).}
\]
Hence, we only need to show that
\[
\int_{D}\left\Vert\left(\frac{F^{\star}H^{3}h}{FF^{\star}}-Q\,\widehat{W}\right)^{'}\right\Vert^{2}\, dA_{\alpha}(z)\leq K(\alpha)^{2}\Vert h\Vert_{\mathcal{D}_{\alpha}}^{2}
\]
 for some $K(\alpha)<\infty.$

\medskip{}
Now
\begin{align*}
\int_{D} & \left\Vert\left(\frac{F^{\star}H^{3}h}{FF^{\star}}-Q\widehat{W}\right)^{'}\right\Vert^{2}\, dA_{\alpha}(z)\\
  & \quad\leq 45\underset{(a')}{\underbrace{\int_{D}\left\Vert\frac{F^{\star}H^{2}H'h}{FF^{\star}}\right\Vert^{2}\, dA_{\alpha}(z)}}+5\underset{(b')}{\underbrace{\int_{D}\left\Vert\frac{F^{\star}H^{3}h'}{FF^{\star}}\right\Vert^{2}\, dA_{\alpha}(z)}}\\
 & \qquad+5\underset{(c')}{\underbrace{\int_{D}\left\Vert\frac{F^{\star}H^{3}h'F'F^{\star}}{\left(FF^{\star}\right)^{2}}\right\Vert^{2}\, dA_{\alpha}(z)}}+5\underset{(d')}{\underbrace{\int_{D}\left\Vert Q'\; \widehat{W}\right\Vert^{2}\, dA_{\alpha}(z)}}\\
 & \qquad+5\underset{(e')}{\underbrace{\int_{D}\Vert Q\left(\widehat{W}\right)^{'}\Vert^{2}\, dA_{\alpha}(z)}}.
\end{align*}

Then
\begin{align*}
(a')=\int_{D}\left\Vert\frac{F^{\star}H^{2}H'h}{FF^{\star}}\right\Vert^{2}\, dA_{\alpha}(z) & =\int_{D}\left\Vert\frac{F^{\star}}{\sqrt{FF^{\star}}}\frac{H}{\sqrt{FF^{\star}}}H\, H\,'h\right\Vert^{2}\, dA_{\alpha}(z)\\
 & \leq\int_{D}\left\Vert H\,'h\right\Vert^{2}\, dA_{\alpha}(z)\\
 & \leq4\,\left\Vert M_{H}\right\Vert_{B(\mathcal{D}_{\alpha})}^{2}\,\left\Vert h\right\Vert_{\mathcal{D}_{\alpha}}^{2}\quad \text{by Lemma 5.}
\end{align*}

\[
(b')=\,\int_{D}\left\Vert\frac{F^{\star}H^{3}h'}{FF^{\star}}\right\Vert^{2}\, dA_{\alpha}(z)\leq\int_{D}\left\Vert h'\right\Vert^{2}\, dA_{\alpha}(z)\leq\left\Vert h\right\Vert_{\mathcal{D}_{\alpha}}^{2}.
\]

\begin{align*}
(c')=\int_{D}\left\Vert\frac{F^{\star}H^{3}hF'F^{\star}}{\left(FF^{\star}\right)^{2}}\right\Vert^{2}\, dA_{\alpha}(z) & =\int_{D}\left\Vert\frac{F^{\star}F'F^{\star}}{\sqrt{FF^{\star}}}\frac{H^{2}}{FF^{\star}}\frac{H}{\sqrt{FF^{\star}}}\, h\right\Vert^{2}\, dA_{\alpha}(z)\\
 & \leq\int_{D}\left\Vert\frac{F^{\star}F'F^{\star}}{\sqrt{FF^{\star}}}\, h\right\Vert^{2}\, dA_{\alpha}(z)\\
 & \leq\int_{D}\left\Vert F'^{\star}h\right\Vert^{2}\, dA_{\alpha}(z)\leq4\,\Vert h\Vert_{\mathcal{D}_{\alpha}}^{2}.
\end{align*}
 We use condition (a) and Lemma 3 to estimate $(e')$.

\begin{align*}
(e') & =\int_{D}\left\Vert Q\left(\widehat{W}\right)^{'}\right\Vert^{2}\, dA_{\alpha}(z)\\
 & \leq \int_{D}\left\Vert\left(\widehat{W}\right)^{'}\right\Vert^{2}\, dA_{\alpha}(z) & (\mbox{ since } \Vert Q(z) \Vert_{B(l^2)}\leq 1)\\
 & \leq \left(\frac{23}{\alpha}\right)^2\int_{D}\left\Vert\frac{Q^{\star}F'^{\star}H^{3}h}{\left(FF^{\star}\right)^{2}}\right\Vert^{2}\, dA_{\alpha}(z) & \left(\mbox{by Lemma 3}\right)\\
 & \leq4\,\left(\frac{23}{\alpha}\right)^2\Vert h\Vert_{\mathcal{D}_{\alpha}}^{2}.
\end{align*}
\noindent 
So we only need to estimate $(d')$. For this, we have
\begin{align*}
(d')=
\int_{D}\left\Vert Q^{'}\, \widehat{W}\right\Vert^{2}\, dA_{\alpha}(z)&\leq2\underset{(f')}{\underbrace{\int_{D}\left\Vert Q^{'} \; \widehat{W}-Q^{'}\; \widetilde{\widehat{W}}\right\Vert^{2}\, dA_{\alpha}(z)}}\\&\;+2\int_{D}\left\Vert Q^{'}\; \widetilde{\widehat{W}}\right\Vert^{2}\, dA_{\alpha}(z),
\end{align*}
 where $\widetilde{\widehat{W}}(z)=\int_{-\pi}^{\pi}\frac{1-\left|z\right|^{2}}{\left|1-e^{-it}z\right|}\,\widehat{W}(e^{it})\, d\sigma(t)$
is the harmonic extension of $\widehat{W}$ from $\partial\mathbb{D}$
to $\mathbb{D}$.\\
Also,
\[
\int_{D}\Vert Q^{'}\widetilde{\widehat{W}}\Vert^{2}\, dA_{\alpha}(z)\leq8\,\Vert\widetilde{\widehat{W}}\Vert_{\mathcal{HD_{\alpha}}}^{2}.
\]
 Also, Lemmas 10 and 11 of {[}KT{]} imply that there is a $C_1<\infty$, independent of $W$ and $\alpha$, satisfying
 \[
\Vert\underline{\widetilde{\widehat{W}}}\Vert_{\mathcal{HD}_{\alpha}}^{2}\leq C_{1}\Vert W\Vert_{A_{\alpha}}^{2}.
\]
 But, as we showed above
\[
\Vert W\Vert_{A_{\alpha}}^{2}=\int_{D}\left\Vert\frac{Q^{\star}F^{'\star}H^{3}h}{\left(FF^{\star}\right)^{2}}\right\Vert^{2}\, dA_{\alpha}(z)\leq\int_{D}\Vert F^{'\star}h\Vert^{2}\, dA_{\alpha}(z)\leq4\,\Vert h\Vert_{\mathcal{D}_{\alpha}}^{2}.
\]
Thus,
\[
\int_{D}\left\Vert Q^{'}\widetilde{\widehat{W}}\right\Vert^{2}\, dA_{\alpha}(z)\leq C_{2}\,\Vert h\Vert_{\mathcal{D}_{\alpha}}^{2},
\]
 where $C_2< \infty$ is independent of $W$ and $\alpha$.
Now we are just left with estimating $(f')$. We have
\begin{align*}
&(f')  =\int_{D}\left\Vert Q^{'}\; \widehat{W}-Q^{'}\; \widetilde{\widehat{W}}\right\Vert^{2}\, dA_{\alpha}(z)\\
 & \; =\int_{D}\left\Vert Q^{'}\bigg[-\frac{1}{\pi}\int_{D}\frac{W(u)}{u-z}dA(u)-\int_{-\pi}^{\pi}\frac{1-\left|z\right|^{2}}{\left|1-e^{-it}z\right|}\widehat{W}(e^{it})d\sigma(t)\bigg]\right\Vert^{2}dA_{\alpha}(z)\\
 & \; =\frac{1}{\pi^{2}}\int_{D}\Vert Q^{'}\int_{D}W(u)\bigg[\frac{1}{u-z}+\int_{-\pi}^{\pi}\frac{1-\left|z\right|^{2}}{\left|1-e^{-it}z\right|}\, e^{-it}\frac{1}{1-ue^{-it}}\, d\sigma(t)\bigg]\\
 & \qquad \qquad \qquad \qquad dA(u)\Vert^{2}dA_{\alpha}(z)\\
 & \; =\frac{1}{\pi^{2}}\int_{D}\left\Vert Q^{'}\int_{D}W(u)\left[\frac{1}{u-z}+\frac{\bar{z}}{1-u\bar{z}}\right]dA(u)\right\Vert^{2}\, dA_{\alpha}(z)\\
  & \; =\frac{1}{\pi^{2}}\int_{D}\left\Vert Q'\int_{D}W(u)\left[\frac{1-\left|z\right|^{2}}{(u-z)(1-u\bar{z\,})}\right]dA(u)\right\Vert^{2}\, dA_{\alpha}(z)\\
 & \; =\frac{1}{\pi^{2}}\int_{D}\left\Vert Q'(z)\,(1-\vert z\vert^{2})\, T(W)(z)\right\Vert^{2}\, dA_{\alpha}(z)\\
 & \; \leq\frac{\left\Vert M_{Q}\right\Vert^{2}}{\pi^{2}}\Vert T(W)\Vert_{A_{\alpha}}^{2}\;\;\;\;\text{ by Lemma  4 }\\
     \end{align*}
 \begin{align*}
 & \; \leq \frac{256}{\alpha^4}\left\Vert M_{Q}\Vert^{2}\,\right\Vert W\Vert_{A_{\alpha}}^{2}\;\;\text{ by Lemma  2}.\\
 & \; \leq \frac{1024}{\alpha^4}\left\Vert M_{Q}\right\Vert^{2}\,\Vert h\Vert_{\mathcal{D}_{\alpha}}^{2}\\
 \end{align*}

By Lemma 9 of [KT], we have $\Vert M_Q \Vert_{B(\oplus \mathcal{HD}_{\alpha})}\leq \sqrt{86}$.
Combining all these pieces, we see that in the smooth case
\[
\left\Vert \underline{u}_{h}\right\Vert _{\mathcal{D}_{\alpha}}^{2}\leq K(\alpha)^2\left\Vert h\right\Vert _{\mathcal{D}_{\alpha}}^{2},
\]
 where $K(\alpha)=K_1 \Vert M_H  \Vert_{B(\mathcal{D_{\alpha}})}+\frac{K_2}{\alpha^2}$, where $K_1<\infty$ and $K_2<\infty$ are constants independent of $h$,\; $\epsilon$  and $\alpha$.

\medskip{}
By the proof of Theorem 1 in the smooth case, we have
$$
M_{H_{r}^3}M_{H_{r}^3}^{\star}\leq K(\alpha)^{2}M_{F_{r}}^{R}(M_{F_{r}}^{R})^{\star}\;\; \text{for }0\leq r<1,$$ where $F_r(z)=F(rz)\;\;  \text{for all}\;\;  z\in \mathbb{D}$.

\noindent Using a commutant lifting theorem for multipliers on weighted Dirichlet space from [KT]\; (see [BTV] for details), there exists $G_{r}\in\mathcal{M}(\mathcal{D}_{\alpha},\,\underset{1}{\overset{\infty}{\oplus}}\,\mathcal{D}_{\alpha}$)
so that $M_{F_{r}}^{R}M_{G_{r}}^{C}=M_{H_{r}^{3}}$ and $\Vert M_{G_{r}}^{C}\Vert\leq K(\alpha)$. The reader should note that such contraction always can be found for the multiplier algebra on reproducing kernel Hilbert spaces with complete Nevanlinna-Pick kernels.  
Then $M_{F_{r}}^{R}\rightarrow M_{F}^{R}$ and $M_{H_{r}^3}\rightarrow M_{H^3}$
as $r\uparrow1$ in the $\star-$strong topology.

\medskip{}
By compactness, we may choose a net with $G_{r_{\alpha}}^{\star}\rightarrow G^{\star}$
as $r_{\alpha}\rightarrow1^{-}.$ Since the multiplier algebra (as
operators) is WOT closed, $G\in\mathcal{M}(\mathcal{D}_{\alpha},\overset{\infty}{\underset{1}{\oplus}}\,\mathcal{D}_{\alpha}).$
Also, since $F_{r_{\alpha}}^{\star}\overset{s}{\rightarrow}F^{\star},$
we get $M_{H_{r}^3}^{\star}=M_{G_{r}}^{\star C}M_{F_{r}}^{\star R}\,\,\overset{WOT}{\rightarrow}\, M_{G}^{\star C}M_{F}^{\star R}$
and so $M_{F}^{R}M_{G}^{C}=M_{H^{3}}$ with entries of $G$ in $\mathcal{M}(\mathcal{D}_{\alpha})$
and $\Vert M_{G}^{C}\Vert\leq K(\alpha).$

This ends our proof.

\bigskip{}

\end{document}